\documentclass[11pt]{amsart}
\usepackage{amsthm,amsmath,amssymb}
\usepackage{stmaryrd,mathrsfs}
\usepackage{enumerate}
\usepackage[T1]{fontenc}
\usepackage{esint}
\usepackage{tikz}
\usepackage{hyperref}
\usepackage{stackengine}
\usepackage{verbatim}
\stackMath
\allowdisplaybreaks
\usepackage{geometry}
\usepackage{hyperref}
	\hypersetup{colorlinks, breaklinks,
            	linkcolor = blue,
				urlcolor = blue,
				anchorcolor = blue,
				citecolor = blue}

\newcommand{\BMO}[0]{\operatorname{BMO}}

\DeclareMathOperator*{\esssup}{ess\;sup}
\DeclareMathOperator*{\essinf}{ess\;inf}

\newcommand{\R}{\mathbb{R}}

\swapnumbers \numberwithin{equation}{section}

\theoremstyle{plain}
\newtheorem{theorem}[equation]{Theorem}

\newtheorem{lemma}[equation]{Lemma}

\theoremstyle{definition}
\newtheorem{definition}[equation]{Definition}

\newtheorem{remark}[equation]{Remark}

\makeatletter
\@namedef{subjclassname@2020}{
  \textup{2020} Mathematics Subject Classification}
 \def\@textbottom{\vskip \z@ \@plus 1pt}
 \let\@texttop\relax
\makeatother

\begin{document}

\title[Factorization of weights]{Factorizations for variable exponent Muckenhoupt weights}

\author[S. Lappas and T. Oikari]{Stefanos Lappas and Tuomas Oikari}
\address{(S. Lappas) Department of Mathematical Analysis, Faculty of Mathematics and Physics, Charles University, Sokolovsk\'a 83, 186 75 Praha 8, Czech Republic \newline and \newline Mathematisches Institut, Rheinische Friedrich-Wilhelms-Universit{\"a}t Bonn, Endenicher Allee 60, 53115 Bonn, Germany }
\email{stefanos.lappas@matfyz.cuni.cz; lappas@math.uni-bonn.de;  vlappas@hotmail.com}

\address{(T. Oikari) Aalto University, Department of Mathematics and Systems Analysis, P.O. Box 11100, FI-00076 Aalto, Finland }
\email{tuomas.oikari@gmail.com}






\begin{abstract}
Given two variable exponent Muckenhoupt weights $w\in \mathcal{A}_{p(\cdot)}$ and $w_1\in \mathcal{A}_{p_1(\cdot)}$, we prove that for all small enough $\theta>0,$ there holds that $w_0\in \mathcal{A}_{p_0(\cdot)},$ where the weight is determined by $w = w_0^{1-\theta}w_1^{\theta}$ and exponent of the weight class by $1/p(\cdot) = (1-\theta)/p_0(\cdot) + \theta/p_1(\cdot).$
The proof is based on a recent reverse H\"older's inequality for variable exponent Muckenhoupt weights 
by Cruz-Uribe and Penrod. We upgrade these factorizations to the restricted range context by using a recent transformation formula due to Nieraeth. Then, following an extrapolation of compactness scheme by Hyt\"onen and Lappas, we provide an alternative proof of the recent extrapolation of compactness results of Lorist and Nieraeth in the context of weighted variable exponent Lebesgue spaces.
\end{abstract}

\maketitle


\section{Introduction}

The aim of this paper is to prove factorization results for variable exponent Muckenhoupt weights in the diagonal, off-diagonal and limited range contexts (see Theorems \ref{thm:fact.main1} and \ref{thm:fact.main2} in the body-text). 
These new factorizations lead, by an argument due to the first named author and Hyt\"onen \cite{HL}, to the compactness extrapolation results, Theorems \ref{thm:main1} and \ref{thm:main2}, on weighted variable exponent Lebesgue spaces $L^{p(\cdot)}_w$ (we recall their definition soon). 
The proof we present is hands-on and elementary. However, we also note that these extrapolation results (including concrete examples to which they are applicable) follow as special cases from Lorist and Nieraeth's work \cite[Sections 1 and 6]{LN} on extrapolation on Banach function spaces. Factorization results for weights are however of independent value and, again, provide a more hands-on approach to such results.  
For an extensive treatment of variable exponent Lebesgue spaces, we refer to the books of Cruz-Uribe, Fiorenza \cite{CF} and of Diening, Harjulehto, H{\"a}st{\"o}, R\r{u}\v{z}i\v{c}ka \cite{DHHR}.
To be able to formulate our results we begin by gathering below some basic definitions and notation.
\begin{itemize}
  \item The class $\mathcal{P}(\mathbb{R}^d)$ stands for measurable functions $p(\cdot):\mathbb{R}^d\rightarrow[1,\infty],$ this is the collection of all variable exponents.
  \item 
Given $p(\cdot)\in \mathcal{P}(\mathbb{R}^d)$, the variable exponent Lebesgue space $L^{p(\cdot)}(\R^d)$ is defined through the norm
\begin{equation*}
  \|f\|_{L^{p(\cdot)}(\R^d)}:=\inf\{\lambda>0:\rho_{p(\cdot)}(f/\lambda)\leq 1\},
\end{equation*}
where the modular is defined through
\begin{equation*}
  \rho_{p(\cdot)}(f):=\int_{\R^d\setminus\{p(\cdot)=\infty\}}|f(x)|^{p(x)}dx+\|f\|_{L^{\infty}(\{p(\cdot)=\infty\})}.
\end{equation*}
\item We abbreviate $\|f\|_{L^{p(\cdot)}(\R^d)} = \|f\|_{p(\cdot)}.$
\item  A weight $w$ is an almost everywhere positive function and is treated as a multiplier; we define the weighted space $L^{p(\cdot)}_w$ through the norm
$$
\|f\|_{L^{p(\cdot)}_w}:= \|fw\|_{p(\cdot)}.
$$
\item Given $p(\cdot)\in\mathcal{P}(\mathbb{R}^d)$, the conjugate exponent $p'(\cdot)$ is given by 
$1/p(\cdot)+ 1/p'(\cdot)=1.$
  \item Given $p(\cdot)\in\mathcal{P}(\mathbb{R}^d)$, we denote
    $p_{-}=\essinf_{x\in \mathbb{R}^d}p(x)$ and $p_{+}=\esssup_{x\in \mathbb{R}^d}p(x).$
  \item The harmonic mean $p_E$ on a set $E$ of positive and finite Lebesgue measure is defined by
\begin{equation*}
  \frac{1}{p_E}:=\displaystyle\stackinset{c}{}{c}{}{-\mkern4mu}{\displaystyle\int_E}\;\frac{1}{p(x)}dx=\frac{1}{|E|}\int_E\;\frac{1}{p(x)}dx.
\end{equation*}
  \item 
We drop the ambient space/dimension $\mathbb{R}^d$ from the notation whenever convenient, as it plays no particular role in our considerations.
\end{itemize}
Common background smoothness assumptions on the variable exponents are the following local and asymptotic H\"older continuity classes. We denote
\begin{itemize}
 \item\label{cond1} $p(\cdot)\in LH_0,$ if there exists $C_0>0$ such that, whenever $|x-y|<\tfrac{1}{2}$, then
    \begin{equation*}
        |p(x)-p(y)|\leq\frac{C_0}{-\log(|x-y|)};
    \end{equation*}
    \item\label{cond2} $p(\cdot)\in LH_{\infty},$ if there exist $p_{\infty}\in\R$ and $C_\infty>0$ such that
    \begin{equation*}
        |p(x)-p_{\infty}|\leq\frac{C_{\infty}}{\log(e+|x|)};
    \end{equation*}
    \item $LH := LH_0\cap LH_{\infty}.$
\end{itemize}
Notice that if $p(\cdot)\in LH$ then $p_{+}<\infty$. Moreover, it follows from \cite[Proposition 2.3]{CF} that $p(\cdot)\in LH\iff 1/p(\cdot)\in LH.$

\subsection{Full range factorization and extrapolation} 
We begin by stating the factorization results (see also Subsection \ref{subs:fact. and extr.}) which are proved in Section \ref{sect:fact}. As easy corollaries, we obtain the extrapolation of compactness results, which are also stated in these subsections, while their proofs are given in Section \ref{sect:last}. Without much added difficulty, we work directly in the off-diagonal setting, meaning that in what follows we do not only consider $\gamma= 0$ but also $\gamma\in (0,1).$ 
Let $\gamma\in [0,1)$ be fixed and we denote
\begin{align*}
  \mathcal{E}_{(1,\infty)}^0 &:= \left\{ p(\cdot)\in \mathcal{P}\cap LH: 1<p_-\leq p_+<\infty \right\}, \\
\notag  \mathcal{E}^{\gamma}_{(1,\infty)} &:= \left\{ \big(p(\cdot),q(\cdot)\big) \in \mathcal{E}_{(1,\infty)}^0\times \mathcal{E}_{(1,\infty)}^0 : 1/p(\cdot) - 1/q(\cdot) = \gamma  \right\}.
\end{align*}
We consider the following class of variable exponent Muckenhoupt weights. 

\begin{definition}[\cite{CW2017}, Definition 2.10]\label{def:variableApq} Let $\gamma\in[0,1)$ and $p(\cdot),q(\cdot)\in\mathcal{P}$ be such that $1/p(\cdot)-1/q(\cdot)=\gamma.$ Then,
a weight $w$ belongs to the class
$\mathcal{A}_{p(\cdot),q(\cdot)}^{\gamma},$ if 
\begin{equation*}
  [w]_{\mathcal{A}_{p(\cdot),q(\cdot)}^{\gamma}}:=\sup_Q |Q|^{\gamma-1}\|w\chi_Q\|_{q(\cdot)}\|w^{-1}\chi_Q\|_{p'(\cdot)}<\infty,
\end{equation*}
where the supremum is taken over all cubes $Q\subset\R^d$.
\end{definition}
When $\gamma = 0$ we understand that $p(\cdot)\in\mathcal{E}_{(1,\infty)}^0$ is equivalent to $(p(\cdot),p(\cdot))\in\mathcal{E}_{(1,\infty)}^0.$ 

Our first factorization theorem for variable exponent Muckenhoupt weights reads as follows.

\begin{theorem}[Full range factorization]\label{thm:fact.main1} Let $\gamma \in [0,1)$ be fixed and 
\begin{itemize}
  \item $\big(p_1(\cdot),q_1(\cdot)\big)\in \mathcal{E}^{\gamma}_{(1,\infty)}$ and  $w_1\in\mathcal{A}_{p_1(\cdot),q_1(\cdot)}^{\gamma},$ 
  \item $\big(p(\cdot),q(\cdot)\big)\in \mathcal{E}^{\gamma}_{(1,\infty)}$ and  $w\in\mathcal{A}_{p(\cdot),q(\cdot)}^{\gamma}.$ 
\end{itemize}
Given any $\theta\in (0,1)$ define the variable exponents $(p_0(\cdot),q_0(\cdot))$ and the weight $w_0$ through
\begin{equation}\label{eq:thm:fact.1:2}
  \left( \frac{1}{p(\cdot)}, \frac{1}{q(\cdot)} \right) =   \left(\frac{1-\theta}{p_0(\cdot)} + \frac{\theta}{p_1(\cdot)} ,\frac{1-\theta}{q_0(\cdot)} + \frac{\theta}{q_1(\cdot)}  \right),\qquad
  w=w_0^{1-\theta}w_1^{\theta}.
\end{equation}
Then, there exists $\eta\in (0,1)$ so that for all $\theta\in(0,\eta)$
there holds that 
\begin{itemize}
  \item $\big(p_0(\cdot),q_0(\cdot)\big)\in \mathcal{E}^{\gamma}_{(1,\infty)}$ and  $w_0\in\mathcal{A}_{p_0(\cdot),q_0(\cdot)}^{\gamma}.$ 
\end{itemize}
\end{theorem}

The following extrapolation of compactness result is derived by using Theorem \ref{thm:fact.main1}.

\begin{theorem}[Full range extrapolation of compactness]\label{thm:main1} 
Let $\gamma\in [0,1)$ and $T$ be a linear operator that is 
\begin{itemize}
  \item bounded from $L^{p_0(\cdot)}_{w_0}$ to $L^{q_0(\cdot)}_{w_0}$ for all $\big(p_0(\cdot),q_0(\cdot)\big) \in   \mathcal{E}^{\gamma}_{(1,\infty)}$ and
   for all $w_0\in\mathcal{A}_{p_0(\cdot),q_0(\cdot)}^{\gamma},$ 
  \item compact from $L^{p_1(\cdot)}_{w_1}$ to $L^{q_1(\cdot)}_{w_1}$ for some $\big(p_1(\cdot),q_1(\cdot)\big)\in\mathcal{E}^{\gamma}_{(1,\infty)}$ and
  for some $w_1\in\mathcal{A}_{p_1(\cdot),q_1(\cdot)}^{\gamma}.$ 
\end{itemize}
Then, the operator $T$ is 
\begin{itemize}
  \item compact from $L^{p(\cdot)}_w$ to $L^{q(\cdot)}_w$ for all $(p(\cdot),q(\cdot))\in \mathcal{E}^{\gamma}_{(1,\infty)}$ and for all $w \in\mathcal{A}_{p(\cdot),q(\cdot)}^{\gamma}.$
\end{itemize}
\end{theorem}

Theorem \ref{thm:main1} is in fact a special circumstance in a more general limited range context.

\subsection{Limited range factorization and extrapolation}\label{subs:fact. and extr.}
We turn to the limited range setting. First, given $r<s$ we denote 
\begin{align*}
  \mathcal{E}_{(r,s)}^0 := \left\{ p(\cdot)\in \mathcal{P}\cap LH: r<p_-\leq p_+<s \right\}.
\end{align*}
For $\gamma\in[0,1)$ and $1\leq r_i<s_i\leq\infty$ for $i=1,2$, we denote
\begin{align*}
  \vec{\mathcal{E}}_{(\vec{r},\vec{s})}^{\gamma}:= \left\{(p(\cdot),q(\cdot)) \in  \mathcal{E}_{(r_1,s_1)}^0\times  \mathcal{E}_{(r_2,s_2)}^0:   \frac{1}{p(\cdot)}- \frac{1}{q(\cdot)} = \frac{1}{r_1}- \frac{1}{r_2} = \frac{1}{s_1}- \frac{1}{s_2} =\gamma \right\}.
\end{align*}

\begin{remark}\label{rem:Eempty}
    We understand that whenever we write $ (p(\cdot),q(\cdot)) \in  \vec{\mathcal{E}}_{(\vec{r},\vec{s})}^{\gamma},$ then the condition 
    $$
    \frac{1}{r_1}- \frac{1}{r_2} = \frac{1}{s_1}- \frac{1}{s_2} =\gamma
    $$ automatically holds, since otherwise $\vec{\mathcal{E}}_{(\vec{r},\vec{s})}^{\gamma}=\varnothing$.
\end{remark}

The following off-diagonal limited range variable weight class was introduced in Nieraeth \cite[Theorem A]{N2023}; the special case $p(\cdot)=q(\cdot)$, $r_1=r_2$, and $s_1=s_2$, is also contained in \cite[Definition 3.4]{N2023}. 

\begin{definition}[\cite{N2023}]\label{variableAp,r,s}
Let $\gamma\in[0,1)$ and $1\leq r_i\leq s_i \leq \infty,$ for $i=1,2$. Then, for
$(p(\cdot),q(\cdot))\in  \vec{\mathcal{E}}^\gamma_{(\vec r,\vec s)}$
the class $\mathcal{A}_{(p(\cdot),q(\cdot)),(\vec{r},\vec{s})}$ is defined through the condition
\begin{equation*}
  [w]_{(p(\cdot),q(\cdot)),(\vec{r},\vec{s})}:=\sup_Q|Q|^{-(\frac{1}{r_1}-\frac{1}{s_1})}\|w\chi_Q\|_{\frac{1}{\frac{1}{q(\cdot)}-\frac{1}{s_2}}}\|w^{-1}\chi_Q\|_{\frac{1}{\frac{1}{r_1}-\frac{1}{p(\cdot)}}}<\infty,
\end{equation*}
where the supremum is taken over all cubes $Q\subset\R^d$.

\begin{remark}
While the definition of $[w]_{(p(\cdot),q(\cdot)),(\vec{r},\vec{s})}$ only explicitly depends on $r_1,s_1,s_2$, but not on $r_2$, we note that $r_2$ is actually determined by the other three exponents, by Remark \ref{rem:Eempty}.
\end{remark}

Notice, unlike in the full range context, that the symbol $\mathcal{A}$ is not included in the notation $[w]_{(p(\cdot),q(\cdot)),(\vec{r},\vec{s})},$ following the notation of Nieraeth \cite{N2023}.
\end{definition}

We now state our second factorization result.

\begin{theorem}[Limited range factorization]\label{thm:fact.main2} 
Let $r_i,s_i\in [1,\infty]$ satisfy $r_i<s_i$ for $i=1,2$ and $\gamma\in [0,1)$ be fixed. Let
\begin{itemize}
  \item $\big(p_1(\cdot),q_1(\cdot)\big) \in \vec{\mathcal{E}}_{(\vec{r},\vec{s})}^{\gamma}$ and $w_1\in\mathcal{A}_{(p_1(\cdot),q_1(\cdot)), (\vec{r},\vec{s})},$
  \item $\big(p(\cdot),q(\cdot)\big) \in \vec{\mathcal{E}}_{(\vec{r},\vec{s})}^{\gamma}$ and $w\in\mathcal{A}_{(p(\cdot),q(\cdot)), (\vec{r},\vec{s})}.$
\end{itemize}
Given a fixed $\theta\in (0,1)$ define the variable exponents $(p_0(\cdot),q_0(\cdot))$ and the weight $w_0$ through
\begin{equation*}
  \left( \frac{1}{p(\cdot)}, \frac{1}{q(\cdot)} \right) =   \left(\frac{1-\theta}{p_0(\cdot)} + \frac{\theta}{p_1(\cdot)} ,\frac{1-\theta}{q_0(\cdot)} + \frac{\theta}{q_1(\cdot)}  \right),\qquad
  w=w_0^{1-\theta}w_1^{\theta}.
\end{equation*}
Then, there exists $\eta\in (0,1)$ so that for all $\theta\in (0,\eta)$ there holds that
\begin{itemize}
  \item $\big(p_0(\cdot),q_0(\cdot)\big)\in \vec{\mathcal{E}}_{(\vec{r},\vec{s})}^{\gamma}$ and $w_0\in \mathcal{A}_{(p_0(\cdot),q_0(\cdot)),(\vec{r},\vec{s})}.$
\end{itemize}
\end{theorem}

Theorem \ref{thm:fact.main2} yields the following extrapolation of compactness result.

\begin{theorem}[Limited range extrapolation of compactness]\label{thm:main2} 
Let $\gamma\in[0,1)$ and $r_i,s_i\in [1,\infty]$ satisfying $r_i<s_i$ for $i=1,2$. Let $T$ be a linear operator that is 
\begin{itemize}
  \item bounded from $L^{p_0(\cdot)}_{w_0}$ to $L^{q_0(\cdot)}_{w_0}$ for all $ (p_0(\cdot),q_0(\cdot)) \in   \vec{\mathcal{E}}_{(\vec{r},\vec{s})}^{\gamma}$ and for all $w_0\in\mathcal{A}_{(p_0(\cdot),q_0(\cdot)),(\vec{r},\vec{s})},$ 
  \item compact from $L^{p_1(\cdot)}_{w_1}$ to $L^{q_1(\cdot)}_{w_1}$ for some $(p_1(\cdot),q_1(\cdot))\in \vec{\mathcal{E}}_{(\vec{r},\vec{s})}^{\gamma}$ and \\
  \phantom{compact from $L^{p_1(\cdot)}(w_1)$ to $L^{q_1(\cdot)}(w_1)$ }for some $w_1\in\mathcal{A}_{(p_1(\cdot),q_1(\cdot)),(\vec{r},\vec{s})}.$ 
\end{itemize}
Then, the operator $T$ is
\begin{itemize}
  \item compact from $L^{p(\cdot)}_w$ to $L^{q(\cdot)}_w$ for all $\big(p(\cdot),q(\cdot)\big)\in \vec{\mathcal{E}}_{(\vec{r},\vec{s})}^{\gamma}$ and for all $w\in\mathcal{A}_{(p(\cdot),q(\cdot)),(\vec{r},\vec{s})}.$ 
\end{itemize}
\end{theorem}

The limited range context covers both the diagonal full/limited range and the off-diagonal full range contexts. 
\begin{itemize}
  \item To recover the diagonal full/limited range context, simply set $\gamma = 0$ to obtain $r_1=r_2$, $s_1=s_2$ and $p(\cdot)=q(\cdot)$.
  \item Given $\gamma\in [0,1)$, we take
   $r_1= 1,$ $s_1 =1/\gamma$ and $r_2=1/(1-\gamma),$ $s_2= \infty$
to recover the off-diagonal full range variable weights as in Definition \ref{def:variableApq}: 
$$
[w]_{\mathcal{A}_{p(\cdot),q(\cdot)}^{\gamma}}= [w]_{_{(p(\cdot),q(\cdot)),(\vec{r},\vec{s})}}.
$$
\end{itemize}

\begin{remark}
Earlier versions of Theorems \ref{thm:fact.main1} and \ref{thm:fact.main2} for classical $L^p$ spaces can be found in \cite{HL}. The literature on the factorization of Muckenhoupt weights in the constant-exponent setting has since expanded considerably; see for example, the works by Cao--Olivo--Yabuta \cite{COY}, Hyt\"onen--Lappas \cite{HL2022}, Liu--Wu--Yang \cite{LWY2023}, Lorist-Nieraeth \cite{LN} and Mart\'{i}n-Reyes--Rivera-R\'{i}os \cite{RR2025}.
\end{remark}

\begin{remark}
Theorems \ref{thm:main1} and \ref{thm:main2} have their predecessors for constant exponent $L^p$ spaces in \cite{HL}. We observe the following differences. In the constant exponent case, the initial boundedness assumption is made for some exponent and all weights, while Theorems \ref{thm:main1} and \ref{thm:main2} make the initial assumption for all exponents (in a certain range) and all weights. In the constant exponent case, the boundedness with some exponent (and all weights) implies the boundedness for all exponents (and all weights) by the classical Rubio de Francia extrapolation of boundedness \cite{RdF}. Versions of Rubio de Francia's extrapolation for variable-exponent spaces are due to Cruz-Uribe--Wang \cite{CW2017} (with further extensions by Cao--Mar\'in--Martell \cite{CMM} and Nieraeth \cite{N2023}). 

However, in their results, the starting point of the extrapolation is boundedness for some {\em constant} exponent (and all weights), and variable exponents only appear in the conclusions. Such results would allow one to replace the boundedness assumption for all variable exponents $(p_0(\cdot),q_0(\cdot))$ by some constant exponents $(p_0,q_0)$ in the same class, but it seems open whether it is possible to start from the boundedness for just one non-constant variable exponent pair $(p_0(\cdot),q_0(\cdot))$.
We do not seem to have an analogue of Rubio de Francia's extrapolation theorem on weighted variable exponent Lebesgue spaces. Furthermore, we point out that Cruz-Uribe--Wang (see \cite[Theorem 2.7]{CW2017}) proved that operators satisfying the hypotheses of the extrapolation theorem for Muckenhoupt $A_p$ weights are bounded on weighted variable Lebesgue spaces. Similar results hold in the off-diagonal and limited range setting (see \cite[Theorems 2.11 and 2.14]{CW2017}). For further extensions of \cite{CW2017} we refer to the recent works by Cao--Mar\'in--Martell \cite{CMM} and Nieraeth \cite{N2023}. 
\end{remark}

\subsection{Acknowledgements and complements} This project was initiated in 2021 during our PhD studies at the University of Helsinki by Tuomas Hyt\"onen introducing us to the idea of extending the compactness extrapolation results from \cite{HL} to variable exponent Lebesgue spaces. Back then we got stuck due to the unavailability of appropriate reverse H\"older's inequalities. Luckily for us, these were provided recently by Cruz-Uribe and Penrod \cite[Theorem 1.1]{CP1} which allowed us to pick up and finish the project that we had started nearly four years ago. In between, the abstract compactness extrapolation results of Lorist and Nieraeth \cite{LN} appeared which yield the results of \cite{HL} in the context of variable exponent Lebesgue spaces.


\section{Preliminaries}

We next gather some known results  on variable exponent Lebesgue spaces that we will need.
The following is immediate from the definitions.
\begin{lemma}[\cite{CF}, Proposition 2.18]\label{lem:homog}
Given a measurable function $p(\cdot):\R^d\rightarrow (0,\infty)$ such that $p_{+}<\infty$, there holds that  $\||f|^s\|_{p(\cdot)}=\|f\|^s_{sp(\cdot)},$ for all $s>0.$ 
\end{lemma}

The next lemma addresses the norm of a characteristic function of a cube.
This result is stated in terms of the $\mathcal{A}_{p(\cdot)}$ condition, given in Definition \ref{def:variableApq} for $\gamma=0$. We state this special case here. In particular, given $p(\cdot)\in\mathcal{P}$, we say that $1\in\mathcal{A}_{p(\cdot)}$ if
\begin{equation*}
[1]_{\mathcal{A}_{p(\cdot)}}:=\sup_{Q}|Q|^{-1}\|\chi_Q\|_{p(\cdot)}\|\chi_Q\|_{p'(\cdot)}<\infty.
\end{equation*}

\begin{remark}\label{rem:1inApVar}
The fact that $p(\cdot)\in LH\cap\mathcal{P}$ with $p_{+}<\infty$ implies $[1]_{\mathcal{A}_{p(\cdot)}}<\infty$ can be found in \cite[Proposition 4.57]{CF}. Notice that there this condition is referred to as the $K_0$ condition. Moreover, as explained in \cite[Example 4.59]{CF} the $K_0$ condition is strictly weaker than $LH$ condition.
\end{remark}

\begin{lemma}[\cite{DHHR}, Lemma 4.5.3]\label{lem:char. function}
Let $p(\cdot)\in LH\cap\mathcal{P}$ with $p_{+}<\infty$. Then,
$\|\chi_Q\|_{p(\cdot)}\sim|Q|^{1/p_Q}$,
for every cube $Q\subset\R^d$ and the implicit constants depend on $p(\cdot)$ and $[1]_{\mathcal{A}_{p(\cdot)}}<\infty$.
\end{lemma}

Another basic result we use is H\"older's inequality in the variable exponent setting.

\begin{lemma}[\cite{TWX}, Lemma 6.2]\label{Gen. Holder's ineq.}
Given $p_1(\cdot),\dots,p_m(\cdot)\in\mathcal{P}$, define $p(\cdot)\in\mathcal{P}$ by
\begin{equation*}
  \frac{1}{p(\cdot)}=\sum_{j=1}^m\frac{1}{p_j(\cdot)},
\end{equation*}
where $j=1,\dots,m$. Then, for all $f_j\in L^{p_j(\cdot)}$ and $f_1\cdots f_m\in L^{p(\cdot)}$ one has
\begin{equation}\label{Holder's impl. constant}
  \|f_1\cdots f_m\|_{L^{p(\cdot)}}\lesssim\|f_1\|_{L^{p_1(\cdot)}}\cdots\|f_m\|_{L^{p_m(\cdot)}},
\end{equation}
where the implicit constant depends only on the $p_j(\cdot)$.
\end{lemma}

\begin{remark}
The H\"older's inequality for two variable exponents can be found in \cite[Corollary 2.28]{CF}. This fact together with an induction argument imply \cite[Lemma 6.2]{TWX}. Moreover, from the proofs of \cite[Corollary 2.28]{CF} and \cite[Lemma 6.2]{TWX} one sees that the implicit constant appearing in \eqref{Holder's impl. constant} is at most $5^{m-1}$, where $m\ge2$.
\end{remark}

A key role will be played by the following recent reverse H\"older's inequality for variable exponent Muckenhoupt weights.

\begin{theorem}[\cite{CP1}, Theorem 1.1]\label{thm:RHI}
Let $r(\cdot)\in LH\cap\mathcal{P}$ with $r_{+}<\infty$ and let $u\in \mathcal{A}_{r(\cdot)}.$ Then, there exists an exponent $s>1$ such that for all cubes $Q\subset \R^d$,
\begin{equation*}
  |Q|^{-\frac{1}{s{r_Q}}}\|u\chi_Q\|_{L^{sr(\cdot)}}\lesssim|Q|^{-\frac{1}{r_Q}}\|u\chi_Q\|_{L^{r(\cdot)}},
\end{equation*}
where the implicit constant depends only on $r(\cdot)$. 
\end{theorem}

In fact, we will use the following corollary, which states that the reverse H\"older's inequality does not only hold for the single $s,$ but in a specific way also for all $t\in(1,s).$

\begin{lemma}[\cite{CP1}, Corollary 4.4]\label{lem:RHI}
Let $r(\cdot)\in LH\cap\mathcal{P}$ with $r_{+}<\infty$. If $u\in \mathcal{A}_{r(\cdot)}$, then there exists an exponent $s>1$ such that for all $t\in(1,s)$ and for all cubes $Q\subset\R^d$, 
\begin{equation*}
  |Q|^{-\frac{1}{{tr_Q}}}\|u\chi_Q\|_{L^{tr(\cdot)}}\lesssim|Q|^{-\frac{1}{r_Q}}\|u\chi_Q\|_{L^{r(\cdot)}},
\end{equation*}
where the implicit constant depends only on $r(\cdot)$ and $[1]_{\mathcal{A}_{\widetilde{v}(\cdot)}}<\infty$, where 
\begin{equation*}
  \widetilde{v}(\cdot):=\frac{r(\cdot)}{\frac1t-\frac1s}.
\end{equation*}
\end{lemma}

\begin{remark}
The fact that $[1]_{\mathcal{A}_{\widetilde{v}(\cdot)}}<\infty$ is contained in the proof of \cite[Corollary 4.4]{CP1}.
\end{remark}

To apply the reverse H\"older's inequality we use the following lemma.
\begin{lemma}[\cite{CW2017}, Proposition 5.4]\label{lem:ApqScaling}
Let $\gamma\in [0,1)$ and define $\sigma' = 1/\gamma.$
Let $p(\cdot),q(\cdot)\in\mathcal{P}$ be such that $1/p(\cdot)-1/q(\cdot)=1/\sigma'.$ Then, $w\in\mathcal{A}_{p(\cdot),q(\cdot)}^{\gamma}$ if and only if $w^{\sigma}\in\mathcal{A}_{q(\cdot)/\sigma}$.
\end{lemma}

\section{Factorizations for variable exponent Muckenhoupt weights}\label{sect:fact}

The technical core of the extrapolation of compactness results, Theorems \ref{thm:main1} and \ref{thm:main2}, following the approach to similar results first studied in \cite{HL}, consists of  factorization Theorems \ref{thm:fact.main1} and \ref{thm:fact.main2}, respectively. In this section, we provide the proofs of these factorization results.

\subsection{Full-range factorization}

In this subsection, we present the proof of the factorization Theorem \ref{thm:fact.main1}. Besides the relevance of this full-range result in its own right, it will also be used in the proof of the limited-range variant further below.

\begin{proof}[Proof of Theorem~\ref{thm:fact.main1}]
A fixed choice of $\theta\in(0,1)$ determines the tuple $\big( p_0(\cdot),q_0(\cdot) \big)$ and the weight $ w_0$  through \eqref{eq:thm:fact.1:2}. 

We first check that $\big( p_0(\cdot),q_0(\cdot) \big) \in \mathcal{E}^{\gamma}_{(1,\infty)},$ provided $\theta < \eta$, and $\eta$ is sufficiently small. We impose that 
\begin{align}\label{eq:lem3:eta1}
  \eta < 
  \min\left( \frac{(p_1)_{-}}{p_+}, 1- \frac{1}{p_{-}} ,  \frac{(q_1)_{-}}{q_+}, 1- \frac{1}{q_{-}} \right).
\end{align}
Notice that $\eta\in (0,1)$ is clear.
From the first two entries in the minimum in \eqref{eq:lem3:eta1}, it follows that
\begin{equation*}
  \frac{1-\theta}{p_0(\cdot)}
  =\frac{1}{p(\cdot)}-\frac{\theta}{p_1(\cdot)}
  \begin{cases}
      \displaystyle\geq\frac{1}{p_+}-\frac{\theta}{(p_1)_-}>\frac{1}{p_+}-\frac{\eta}{(p_1)_-}>0, \\
      \displaystyle\leq\frac{1}{p_-}-\frac{\theta}{(p_1)_+}<\frac{1}{p_-}<1-\eta<1-\theta,
  \end{cases}
\end{equation*}
and hence $(p_0)_{-}, ({p_0})_{+} \in (1,\infty).$
Exactly the same argument (using the two rightmost entries inside the minimum bounding $\eta$) shows that $(q_0)_{-}, ({q_0})_{+} \in (1,\infty).$

Next we check that $p_0(\cdot),q_0(\cdot)\in LH.$  Since $(p_0)_+,(q_0)_+<\infty,$ by \cite[Proposition 2.3]{CF} we know that $p_0(\cdot),q_0(\cdot)\in LH$ if and only if $1/p_0(\cdot),1/q_0(\cdot)\in LH.$ Similarly $p_+,q_+,(p_1)_+,(q_1)_+<\infty$ gives $1/p(\cdot),1/q(\cdot),1/p_1(\cdot),1/q_1(\cdot)\in LH.$ Therefore $1/p_0(\cdot),1/q_0(\cdot)$ (by \eqref{eq:thm:fact.1:2}) are linear combinations of functions in $LH.$ Since $LH$ is closed under linear combinations, we get $1/p_0(\cdot),1/q_0(\cdot)\in LH.$ Since $(p_0)_+,(q_0)_+<\infty$, this is further equivalent to $p_0(\cdot),q_0(\cdot)\in LH.$ 
To see that $1/p_0(\cdot)-1/q_0(\cdot) = \gamma,$ notice that
\begin{align*}
  \gamma = \frac{1}{p(\cdot)} - \frac{1}{q(\cdot)} = \frac{1-\theta}{p_0(\cdot)} + \frac{\theta}{p_1(\cdot)} - \frac{1-\theta}{q_0(\cdot)} - \frac{\theta}{q_1(\cdot)} = (1-\theta)\left( \frac{1}{p_0(\cdot)}-\frac{1}{q_0(\cdot)}\right) +\theta\gamma.
\end{align*}
Thus we have shown that $\big( p_0(\cdot),q_0(\cdot) \big) \in \mathcal{E}^{\gamma}_{(1,\infty)},$ provided $\theta<\eta$ and \eqref{eq:lem3:eta1} holds. 

It remains to show that, under a further restriction of $\eta$, there holds that $w_0\in \mathcal{A}_{p_0(\cdot),q_0(\cdot)}^{\gamma}$ for all $\theta<\eta.$ Fix a cube $Q$ and after solving $w_0= w^{1/(1-\theta)}w_1^{-\theta/(1-\theta)}$ (by \eqref{eq:thm:fact.1:2}) write 
\begin{equation*}
\begin{split}
  &\|w_0\chi_Q\|_{q_0(\cdot)}\|w_0^{-1}\chi_{Q}\|_{p_{0}'(\cdot)}
  =\|w^{\frac{1}{1-\theta}}w_1^{-\frac{\theta}{1-\theta}}\chi_Q\|_{q_0(\cdot)}\|w^{-\frac{1}{1-\theta}}w_1^{\frac{\theta}{1-\theta}}\chi_Q\|_{p_{0}'(\cdot)}.
\end{split}
\end{equation*}
The idea is to use H\"older's inequality appropriately in preparation for a later use of the reverse H\"older's inequality. 
Let $\iota>0$ be a positive constant (to be later fixed as a function of the reverse H\"older's inequality) and define $e(\cdot) = e_{\theta}(\cdot)$ and $u(\cdot)=u_{\theta}(\cdot)$ through 
\begin{align}\label{eq:rel:e}
  \frac{1}{q_0(\cdot)} = \frac{1}{ (1+\iota)q(\cdot)} + \frac{\theta}{(1+\iota)p'_1(\cdot)} + \frac{1}{e(\cdot)},
\end{align}
\begin{align}\label{eq:rel:u}
  \frac{1}{p_0'(\cdot)} = \frac{1}{(1+\iota)p'(\cdot)} + \frac{\theta}{(1+\iota)q_1(\cdot)} + \frac{1}{u(\cdot)}.
\end{align}
We next check that if $\theta$ is sufficiently small (for a given fixed $\iota$), then $e_{\theta}(\cdot),u_{\theta}(\cdot)\in\mathcal{E}_{(1,\infty)}^0.$
Solving for $1/q_0(\cdot)$ from \eqref{eq:thm:fact.1:2} and using \eqref{eq:rel:e} we find that 
\begin{align*}
  \frac{1}{(1-\theta)q(\cdot)} - \frac{\theta}{(1-\theta)q_1(\cdot)} = \frac{1}{ (1+\iota)q(\cdot)} + \frac{\theta}{(1+\iota)p'_1(\cdot)} + \frac{1}{e_{\theta}(\cdot)} ,
\end{align*}
and rearranging gives
\begin{align}\label{eq:lem3:limit0}
  \frac{1}{e_{\theta}(\cdot)} = \left( \frac{1}{1-\theta} -\frac{1}{ 1+\iota} \right) \frac{1}{q(\cdot)} - \theta\left(  \frac{1}{(1-\theta)q_1(\cdot)} + \frac{1}{(1+\iota)p'_1(\cdot)}\right),
\end{align}
which shows that 
\begin{align}\label{eq:lem3:limit1}
  \lim_{\theta\to 0^{+}} \frac{1}{e_{\theta}(\cdot)} = \left( 1 -\frac{1}{ 1+\iota} \right) \frac{1}{q(\cdot)}.
\end{align}
Since $\iota>0,$ it is immediate by \eqref{eq:lem3:limit1} that we may choose $\theta$ so small that $e = e_{\theta}(\cdot)\in\mathcal{P}$ with $e_{-}>1$. Clearly also $1/e(\cdot)\in LH$ as a linear combination of elements of $LH$ (this follows by \cite[Proposition 2.3]{CF} and the fact that $p'_1(\cdot)\in LH$ since $p_1(\cdot)\in LH$ and $(p_{1})_{-}>1$). Thus, by $e_+ < \infty$ (clear from \eqref{eq:lem3:limit0}) and \cite[Proposition 2.3]{CF}, it follows that $e(\cdot)\in\mathcal{E}_{(1,\infty)}^0.$
A similar analysis shows that $u(\cdot) = u_{\theta}(\cdot)\in\mathcal{E}_{(1,\infty)}^0,$ provided $\theta$ is sufficiently small. Since $e(\cdot),u(\cdot)\in\mathcal{E}_{(1,\infty)}^0$, we have 
\begin{equation}\label{eq:chiQnorm}
  \|\chi_Q\|_{e(\cdot)} \sim |Q|^{1/e_Q}\qquad\text{and}\qquad\|\chi_Q\|_{u(\cdot)} \sim |Q|^{1/u_Q},
\end{equation}
by Lemma \ref{lem:char. function}.

Now H\"older's inequality (Lemma \ref{Gen. Holder's ineq.}) applied with the variable exponents \eqref{eq:rel:e} shows that 
\begin{equation}\label{eq:befRHIa}
            \|w^{\frac{1}{1-\theta}}w_1^{-\frac{\theta}{1-\theta}}\chi_Q\|_{q_0(\cdot)} \lesssim  \|w^{\frac{1}{1-\theta}}\chi_Q\|_{(1+\iota)q(\cdot)}\|w_1^{-\frac{\theta}{1-\theta}}\chi_Q\|_{\big(\frac{1+\iota}{\theta}\big)p'_1(\cdot)}\|\chi_Q\|_{e(\cdot)}  
\end{equation}
and with the variable exponents \eqref{eq:rel:u} that
\begin{equation}\label{eq:befRHIb}
         \|w^{-\frac{1}{1-\theta}}w_1^{\frac{\theta}{1-\theta}}\chi_Q\|_{p_{0}'(\cdot)} \lesssim \|w^{-\frac{1}{1-\theta}}\chi_Q\|_{(1+\iota)p'(\cdot)}\|w_1^{\frac{\theta}{1-\theta}}\chi_Q\|_{\big(\frac{1+\iota}{\theta}\big)q_1(\cdot)}\|\chi_Q\|_{u(\cdot)} 
\end{equation}

By summing equations \eqref{eq:rel:e} and \eqref{eq:rel:u}, we obtain
\begin{align*}
  1-\gamma  &= \frac{1}{q_0(\cdot)} +  \frac{1}{p_0'(\cdot)} = \frac{1}{1+\iota}\left(\frac{1}{q(\cdot)} + \frac{1}{p'(\cdot)}\right) +  \frac{\theta}{1+\iota}\left( \frac{1}{p_1'(\cdot)} + \frac{1}{q_1(\cdot)}\right) + \frac{1}{e(\cdot)} + \frac{1}{u(\cdot)} \\ 
  &=\frac{1-\gamma}{1+\iota}+  \theta\frac{1-\gamma}{1+\iota} + \frac{1}{e(\cdot)} + \frac{1}{u(\cdot)}.
\end{align*}
After rearranging, this gives 
\begin{align*}
  \frac{1}{e(\cdot)} + \frac{1}{u(\cdot)} = (1-\gamma)\left( 1 - \left( \frac{1}{1+\iota} + \frac{\theta}{1+\iota}\right) \right)  = (1-\gamma)\left( 1 - \frac{1+\theta}{1+\iota} \right),
\end{align*}
and in particular, for all cubes $Q$, there holds
\begin{align}\label{eq:u+e}
  \frac{1}{e_Q} + \frac{1}{u_Q} = (1-\gamma)\left( 1 - \frac{1+\theta}{1+\iota} \right).
\end{align}
Multiplying the bounds \eqref{eq:befRHIa} and \eqref{eq:befRHIb} together, using \eqref{eq:chiQnorm} and \eqref{eq:u+e} and rearranging, we obtain 
\begin{equation}\label{eq:lem3:Y}
    \begin{split}
         &|Q|^{\gamma-1}\|w_0\chi_Q\|_{q_0(\cdot)}\|w_0^{-1}\chi_{Q}\|_{p_{0}'(\cdot)} \\
    & \lesssim |Q|^{-(1+\theta)\frac{1-\gamma}{1+\iota}} \left[ \|w^{\frac{1}{1-\theta}}\chi_Q\|_{(1+\iota)q(\cdot)}  \|w^{-\frac{1}{1-\theta}}\chi_Q\|_{(1+\iota)p'(\cdot)} \right] \\
    &\qquad\qquad\qquad\qquad \times\left[ \|w_1^{\frac{\theta}{1-\theta}}\chi_Q\|_{\big(\frac{1+\iota}{\theta}\big)q_1(\cdot)}\|w_1^{-\frac{\theta}{1-\theta}}\chi_Q\|_{\big(\frac{1+\iota}{\theta}\big)p'_1(\cdot)}  \right]
    \end{split}
\end{equation}
We consider the square bracketed terms separately. 
By Lemma \ref{lem:homog} we write
\begin{align}\label{eq:lem3:X}
  \|w^{\frac{1}{1-\theta}}\chi_Q\|_{(1+\iota)q(\cdot)}  \|w^{-\frac{1}{1-\theta}}\chi_Q\|_{(1+\iota)p'(\cdot)}  = \|w^{\sigma}\chi_Q\|^{\frac{1}{\sigma(1-\theta)}}_{\frac{1+\iota}{1-\theta}\frac{q(\cdot)}{\sigma}}  \|w^{-\sigma}\chi_Q\|^{\frac{1}{\sigma(1-\theta)}}_{\frac{1+\iota}{1-\theta}\frac{p'(\cdot)}{\sigma}}.
\end{align}
Recall that $w\in \mathcal{A}_{p(\cdot),q(\cdot)}^{\gamma}$ if and only if $w^{-1}\in \mathcal{A}_{q'(\cdot),p'(\cdot)}^{\gamma}$ (immediately from the definition).
Now Lemma \ref{lem:ApqScaling} applied with $\sigma = \frac{1}{1-\gamma},$ 
shows that $w^{\sigma} \in \mathcal{A}_{q(\cdot)/\sigma}$ and $w^{-\sigma} \in \mathcal{A}_{p'(\cdot)/\sigma}$. Moreover, notice that $q(\cdot)/\sigma, p'(\cdot)/\sigma\in\mathcal{P}\cap LH$  with $(q(\cdot)/\sigma)_{+},(p'(\cdot)/\sigma)_{+}<\infty$. Therefore the reverse H\"older's inequality, Lemma \ref{lem:RHI}, tells us that there exist some $s_1,s_2>1$ such that 
\begin{align}\label{eq:lem3:XX}
  \|w^{\sigma}\chi_Q\|_{t\frac{q(\cdot)}{\sigma}} \lesssim |Q|^{(\frac1t-1) \sigma/q_Q}  \|w^{\sigma}\chi_Q\|_{\frac{q(\cdot)}{\sigma}},\quad 1<t<s_1,
\end{align}
\begin{align}\label{eq:lem3:XXX}
  \|w^{-\sigma}\chi_Q\|_{t\frac{p'(\cdot)}{\sigma}} \lesssim |Q|^{(\frac1t-1) \sigma/p'_Q} \|w^{-\sigma}\chi_Q\|_{\frac{p'(\cdot)}{\sigma}},\quad 1<t<s_2.
\end{align}
Now we demand from $\theta,\iota$ that
\begin{equation}\label{eq:tChoice}
  1< t:= t(\theta,\iota) := (1+\iota)/(1-\theta) < \min(s_1,s_2).
\end{equation}  
Applying Lemma \ref{lem:homog} and estimates \eqref{eq:lem3:XX} and \eqref{eq:lem3:XXX}, we obtain
\begin{align*}
  \mathrm{RHS}\eqref{eq:lem3:X}^{\sigma(1-\theta)} &= \|w^{\sigma}\chi_Q\|_{t \frac{q(\cdot)}{\sigma}}  \|w^{-\sigma}\chi_Q\|_{t \frac{p'(\cdot)}{\sigma}} \\ 
  &\lesssim |Q|^{ \sigma(\frac1t-1)\left( 1/q_Q + 1/p'_Q\right)}  \|w^{\sigma}\chi_Q\|_{\frac{q(\cdot)}{\sigma}}\|w^{-\sigma}\chi_Q\|_{\frac{p'(\cdot)}{\sigma}} \\ 
  &= |Q|^{ \sigma(\frac1t-1)(1/q_Q+1/p'_Q)}  \|w\chi_Q\|_{q(\cdot)}^{\sigma}\|w^{-1}\chi_Q\|_{p'(\cdot)}^{\sigma} \\
  &\leq |Q|^{\sigma(\frac1t-1)(1-\gamma)}  |Q|^{\sigma(1-\gamma)} [w]_{\mathcal{A}_{p(\cdot),q(\cdot)}^{\gamma}}^{\sigma}  \\
  &= |Q|^{\sigma\frac1t(1-\gamma)}[w]_{\mathcal{A}_{p(\cdot),q(\cdot)}^{\gamma}}^{\sigma}
  = |Q|^{\sigma\frac{1-\theta}{1+\iota}(1-\gamma)}[w]_{\mathcal{A}_{p(\cdot),q(\cdot)}^{\gamma}}^{\sigma},
\end{align*}
where we substituted the value of $t$ from \eqref{eq:tChoice} in the last step. Taking roots of both sides, we have
\begin{equation}\label{eq:RHS<Ap}
  \mathrm{RHS}\eqref{eq:lem3:X} \lesssim 
  |Q|^{\frac{1-\gamma}{1+\iota}}[w]_{\mathcal{A}_{p(\cdot),q(\cdot)}^{\gamma}}^{\frac{1}{1-\theta}}.
 \end{equation}
 
Then, we consider the second bracketed term of \eqref{eq:lem3:Y}. By Lemma \ref{lem:homog}, we write it as 
\begin{align}\label{eq:lem3:Z}
  \|w_1^{\frac{\theta}{1-\theta}}\chi_Q\|_{\big(\frac{1+\iota}{\theta}\big)q_1(\cdot)}\|w_1^{-\frac{\theta}{1-\theta}}\chi_Q\|_{\big(\frac{1+\iota}{\theta}\big)p'_1(\cdot)} =  \left(\|w_1^{\sigma}\chi_Q\|_{\big(\frac{1+\iota}{1-\theta}\big)\frac{q_1(\cdot)}{\sigma}}^{\frac{1}{\sigma(1-\theta)}}\|w_1^{-\sigma}\chi_Q\|_{\big(\frac{1+\iota}{1-\theta}\big)\frac{p'_1(\cdot)}{\sigma}}^{\frac{1}{\sigma(1-\theta)}}\right)^{\theta}.
\end{align}
Notice that $\mathrm{RHS}\eqref{eq:lem3:Z}$ is formally of exactly the same form as $\mathrm{RHS}\eqref{eq:lem3:X}^{\theta}$ but with the variable exponents $q_1(\cdot),p_1'(\cdot)$ in place of $q(\cdot),p'(\cdot).$ Therefore, exactly the same analysis as above yields, for an appropriate choice of $t(\theta,\iota)$ small enough, that
\begin{equation}\label{eq:RHS<ApTheta}
  \mathrm{RHS}\eqref{eq:lem3:Z} \lesssim \left( |Q|^{\frac{1-\gamma}{1+\iota}}[w_1]_{\mathcal{A}_{p_1(\cdot),q_1(\cdot)}^{\gamma}}^{\frac{1}{1-\theta}} \right)^{\theta}.
\end{equation}
Chaining all of the above estimates together, the proof is concluded by
\begin{align*}
  &|Q|^{\gamma-1}\|w_0\chi_Q\|_{q_0(\cdot)}\|w_0^{-1}\chi_{Q}\|_{p_{0}'(\cdot)} \\ 
  &\lesssim |Q|^{-(1+\theta)\frac{1-\gamma}{1+\iota}} \left[ \|w^{\frac{1}{1-\theta}}\chi_Q\|_{(1+\iota)q(\cdot)}  \|w^{-\frac{1}{1-\theta}}\chi_Q\|_{(1+\iota)p'(\cdot)} \right] \\
  &\qquad\qquad \times\left[ \|w_1^{\frac{\theta}{1-\theta}}\chi_Q\|_{\big(\frac{1+\iota}{\theta}\big)q_1(\cdot)}\|w_1^{-\frac{\theta}{1-\theta}}\chi_Q\|_{\big(\frac{1+\iota}{\theta}\big)p'_1(\cdot)}  \right]\qquad\text{by \eqref{eq:lem3:Y}} \\
  &\lesssim  |Q|^{-(1+\theta)\frac{1-\gamma}{1+\iota}}|Q|^{\frac{1-\gamma}{1+\iota}}|Q|^{\theta\frac{1-\gamma}{1+\iota}}[w]_{\mathcal{A}_{p(\cdot),q(\cdot)}^{\gamma}}^{\frac{1}{1-\theta}}[w_1]_{\mathcal{A}_{p_1(\cdot),q_1(\cdot)}^{\gamma}}^{\frac{\theta}{1-\theta}}
  \qquad\text{by \eqref{eq:RHS<Ap} and \eqref{eq:RHS<ApTheta}} \\
  &= [w]_{\mathcal{A}_{p(\cdot),q(\cdot)}^{\gamma}}^{\frac{1}{1-\theta}}[w_1]_{\mathcal{A}_{p_1(\cdot),q_1(\cdot)}^{\gamma}}^{\frac{\theta}{1-\theta}}.\qedhere
\end{align*}
\end{proof}


\subsection{Limited-range factorization}

We turn to the off-diagonal limited range context. In particular, we provide the proof of the second key factorization Theorem \ref{thm:fact.main2} which is essential for the proof of Theorem \ref{thm:main2}.

\begin{proof}[Proof of Theorem \ref{thm:fact.main2}]
The strategy of the proof is to apply a transformation to reduce the situation at hand to the one that was already covered in Theorem \ref{thm:fact.main1}, then apply that theorem, and finally transform back.

Turning to the details, we consider the following maps introduced by Nieraeth \cite{N2023}:
\begin{align*}
  \phi_{r,s}(t) := \frac{t-1/s}{1/r-1/s},\qquad w_{r,s} := w^{1/(1/r-1/s)}
\end{align*}
and define new variable exponents through
\begin{align*}
  \frac{1}{p_{r,s}(\cdot)} := \phi_{r,s}\left(\frac{1}{p(\cdot)}\right),\qquad \frac{1}{p_{r,s}'(\cdot)} = 1- \frac{1}{p_{r,s}(\cdot)}.
\end{align*}
Our standing assumptions state that
\begin{align}\label{eq:lem:main21}
  1/r_1-1/r_2 = 1/s_1 - 1/s_2 = 1/p_1(\cdot) - 1/q_1(\cdot) = \gamma.
\end{align}
Notice that $1/r_1 - 1/s_1 = 1/r_2 - 1/s_2$ by the left-most identity of \eqref{eq:lem:main21}.
Then,
\begin{align*}
  \frac{1}{(p_1)_{r_1,s_1}(\cdot)} -  \frac{1}{(q_1)_{r_2,s_2}(\cdot)} = 
  \frac{1}{1/r_1-1/s_1}\left( \frac{1}{p_1(\cdot)} - \frac{1}{q_1(\cdot)}-\left(\frac{1}{s_1}-\frac{1}{s_2}\right) \right)  =  \frac{\gamma-\gamma}{1/r_1-1/s_1}=0,
\end{align*} 
showing that (the latter identity is checked similarly)
\begin{equation}\label{eq:lem:main2:6}
       \begin{split}
            (p_1)_{r_1,s_1}(\cdot) = (q_1)_{r_2,s_2}(\cdot),\qquad
            p_{r_1,s_1}(\cdot) = q_{r_2,s_2}(\cdot).
       \end{split}
   \end{equation}
Then, a computation, using Lemma \ref{lem:homog}, shows that
\begin{equation}\label{eq:lem:main2:5}
    \begin{split}
          &\Big( \|w_1\chi_Q\|_{\frac{1}{\frac{1}{q_1(\cdot)}-\frac{1}{s_2}}}\|w_1^{-1}\chi_Q\|_{\frac{1}{\frac{1}{r_1}-\frac{1}{p_1(\cdot)}}}  \Big)^{\frac{1}{1/r_1-1/s_1}} \\
    &= 
    \|(w_1)_{r_2,s_2}\chi_Q\|_{\frac{\frac{1}{r_2}-\frac{1}{s_2}}{\frac{1}{q_1(\cdot)}-\frac{1}{s_2}}}\|(w_1)_{r_1,s_1}^{-1}\chi_Q\|_{\frac{\frac{1}{r_1}-\frac{1}{s_1}}{\frac{1}{r_1}-\frac{1}{p_1(\cdot)}}} \\
    &= 
     \|(w_1)_{r_2,s_2}\chi_Q\|_{(q_1)_{r_2,s_2}(\cdot)}\|(w_1)_{r_1,s_1}^{-1}\chi_Q\|_{ (p_1)_{r_1,s_1}'(\cdot)} \\ 
     &= 
   \|(w_1)_{r_1,s_1}\chi_Q\|_{(p_1)_{r_1,s_1}(\cdot)}\|(w_1)_{r_1,s_1}^{-1}\chi_Q\|_{ (p_1)_{r_1,s_1}'(\cdot)},
    \end{split}
\end{equation}
where in the last equation we used $(p_1)_{r_1,s_1}(\cdot) = (q_1)_{r_2,s_2}(\cdot)$ and $(w_1)_{r_2,s_2} = (w_1)_{r_1,s_1}.$
Thus
\begin{align*}
  [(w_1)_{r_1,s_1}]_{\mathcal{A}_{(p_1)_{r_1,s_1}(\cdot)}}^{1/r_1-1/s_1} = [w_1]_{_{(p_1(\cdot),q_1(\cdot)), (\vec{r},\vec{s})}}< \infty.
\end{align*}
Similarly for the weight $w,$ we have 
\begin{align*}
  [w_{r_1,s_1}]_{\mathcal{A}_{p_{r_1,s_1}(\cdot)}}^{1/r_1-1/s_1} = [w]_{_{(p(\cdot),q(\cdot)), (\vec{r},\vec{s})}} < \infty.
\end{align*}

To apply Theorem \ref{thm:fact.main1} with the weights $(w_1)_{r_1,s_1},w_{r_1,s_1}$,  notice that 
$(p_1)_{r_1,s_1}(\cdot),p_{r_1,s_1}(\cdot) \in \mathcal{E}^{0}_{(1,\infty)}.$
Indeed this follows as $\phi_{r_1,s_1}:[1/s_1,1/r_1]\to[0,1]$ continuously and $p_1(\cdot),p(\cdot)\in\mathcal{E}^{0}_{(r_1,s_1)}.$
Therefore, by Theorem \ref{thm:fact.main1}, there exists some $\eta=\eta\left(p_{r_1,s_1}(\cdot), (p_1)_{r_1,s_1}(\cdot), w_{r_1,s_1}, (w_1)_{r_1,s_1}\right)\in (0,1)$ such that the following holds:
whenever $\theta\in (0,\eta),$ and $\mu=\mu_\theta$ and $e_0(\cdot)=e_\theta(\cdot)$ are defined by
\begin{equation}\label{eq:e0mu}
  \frac{1}{p_{r_1,s_1}(\cdot)}=\frac{1-\theta}{e_0(\cdot)}+\frac{\theta}{(p_1)_{r_1,s_1}(\cdot)},\qquad  w_{r_1,s_1}=(\mu)^{1-\theta}(w_1)_{r_1,s_1}^{\theta},
\end{equation}
then
\begin{align*}
  \mu \in \mathcal{A}_{e_0(\cdot)},\qquad  e_0(\cdot)\in \mathcal{E}^{0}_{(1,\infty)}.
\end{align*}
In particular, there holds that 
\begin{align*}
    [\mu]_{ \mathcal{A}_{e_0(\cdot)}}\lesssim [w_{r_1,s_1}]_{\mathcal{A}_{p_{r_1,s_1}(\cdot)}}^{\frac{1}{1-\theta}}[(w_1)_{r_1,s_1}]_{\mathcal{A}_{(p_1)_{r_1,s_1}(\cdot)}}^{\frac{\theta}{1-\theta}}.
\end{align*}

Now define the variable exponents $(p_0(\cdot), q_0(\cdot))$ and the weight $w_0$  through the relations
\begin{align}\label{eq:lem:main2:3}
  w_0 = \mu^{1/r_1-1/s_1},\qquad (p_0)_{r_1,s_1}(\cdot) = e_0(\cdot),\qquad \frac{1}{q(\cdot)} = \frac{1-\theta}{q_0(\cdot)} + \frac{\theta}{q_1(\cdot)}. 
\end{align}

Notice that there exist some constants $c = c(r_1,s_1)$ and $c'=c'(r_1,s_1)$ such that $1/u(\cdot) = c/u_{r_1,s_1}(\cdot) + c',$ for all variable exponents $u(\cdot).$ Therefore, 
\begin{equation}\label{eq:lem:main2:7}
    \begin{split}
        \frac{1}{p(\cdot)}= \frac{c}{p_{r_1,s_1}(\cdot)} + c' &= c\left(\frac{1-\theta}{e_0(\cdot)}+\frac{\theta}{(p_1)_{r_1,s_1}(\cdot)}\right) + c'
        \qquad\text{by \eqref{eq:e0mu}}
        \\ 
    &= c\left(\frac{1-\theta}{(p_0)_{r_1,s_1}(\cdot)}+\frac{\theta}{(p_1)_{r_1,s_1}(\cdot)}\right) + c' 
    \qquad\text{by \eqref{eq:lem:main2:3}}\\ 
    &= (1-\theta) \left(  \frac{c}{(p_0)_{r_1,s_1}(\cdot)} + c' \right) + \theta\left(  \frac{c}{(p_1)_{r_1,s_1}(\cdot)} + c' \right) \\ 
    &= \frac{1-\theta}{p_0(\cdot)}+\frac{\theta}{p_1(\cdot)}.
    \end{split}
\end{equation}

Then, we check that $p_0(\cdot)\in \mathcal{E}_{(r_1,s_1)}^0$ and $q_0(\cdot)\in\mathcal{E}_{(r_2,s_2)}^0.$
Since $(p_0)_{r_1,s_1}(\cdot) = e_0(\cdot)\in \mathcal{E}_{(1,\infty)}^0,$ and $\phi_{r_1,s_1}:[1/s_1,1/r_1] \to [0,1]$ continuously, it is clear that $p_0(\cdot)\in \mathcal{E}_{(r_1,s_1)}^0.$ By assumption we know that $q(\cdot),q_1(\cdot)\in\mathcal{E}_{(r_2,s_2)}^0;$ therefore, if $\theta$ is sufficiently small, also $q_0(\cdot)\in\mathcal{E}_{(r_2,s_2)}^0.$

To finish the  verification that $\big(p_0(\cdot),q_0(\cdot)\big)\in \vec{\mathcal{E}}_{(\vec{r},\vec{s})}^{\gamma},$
we use the right-most identity of \eqref{eq:lem:main2:3} and \eqref{eq:lem:main2:7} to obtain
\begin{align*}
  \gamma = \frac{1}{p(\cdot)} - \frac{1}{q(\cdot)} = \frac{1-\theta}{p_0(\cdot)} + \frac{\theta}{p_1(\cdot)} - \frac{1-\theta}{q_0(\cdot)} - \frac{\theta}{q_1(\cdot)} = (1-\theta)\left( \frac{1}{p_0(\cdot)}-\frac{1}{q_0(\cdot)}\right) +\theta\gamma,
\end{align*}
which immediately gives that $1/p_0(\cdot)-1/q_0(\cdot) = \gamma.$ 

It remains to verify that $w_0\in \mathcal{A}_{(p_0(\cdot),q_0(\cdot)),(\vec{r},\vec{s})}.$ First, notice that, by the same computation as \eqref{eq:lem:main2:7}, it is clear that
\begin{align*}
  \frac{1}{q_{r_2,s_2}(\cdot)} =  \frac{1-\theta}{(q_0)_{r_2,s_2}(\cdot)} +  \frac{\theta}{(q_1)_{r_2,s_2}(\cdot)}.
\end{align*}
Together with \eqref{eq:lem:main2:6}, this shows that 
\begin{align*}
  \frac{1}{p_{r_1,s_1}(\cdot)}= \frac{1}{q_{r_2,s_2}(\cdot)} =  \frac{1-\theta}{(q_0)_{r_2,s_2}(\cdot)} + \frac{\theta}{(q_1)_{r_2,s_2}(\cdot)} =  \frac{1-\theta}{(q_0)_{r_2,s_2}(\cdot)} + \frac{\theta}{(p_1)_{r_1,s_1}(\cdot)}.
\end{align*}
Since the solution $(q_0)_{r_2,s_2}(\cdot)$ to this equation is unique, and we know that $(p_0)_{r_1,s_1}(\cdot)$ is a solution (by repeating the computation \eqref{eq:lem:main2:7}), it follows that $(q_0)_{r_2,s_2}(\cdot) = (p_0)_{r_1,s_1}(\cdot).$ Now, 
repeating the computation \eqref{eq:lem:main2:5} and recalling definitions, we find
\begin{align*}
  \Big( \|w_0\chi_Q\|_{\frac{1}{\frac{1}{q_0(\cdot)}-\frac{1}{s_2}}}\|w_0^{-1}\chi_Q\|_{\frac{1}{\frac{1}{r_1}-\frac{1}{p_0(\cdot)}}}  \Big)^{\frac{1}{1/r_1-1/s_1}} 
  &= 
  \|(w_0)_{r_2,s_2}\chi_Q\|_{\frac{\frac{1}{r_2}-\frac{1}{s_2}}{\frac{1}{q_0(\cdot)}-\frac{1}{s_2}}}\|(w_0)_{r_1,s_1}^{-1}\chi_Q\|_{\frac{\frac{1}{r_1}-\frac{1}{s_1}}{\frac{1}{r_1}-\frac{1}{p_0(\cdot)}}} \\
  &= 
  \|\mu\chi_Q\|_{(q_0)_{r_2,s_2}(\cdot)}\|\mu^{-1}\chi_Q\|_{ (p_0)_{r_1,s_1}'(\cdot)} \\ 
  &= 
  \|\mu\chi_Q\|_{(p_0)_{r_1,s_1}(\cdot)}\|\mu^{-1}\chi_Q\|_{ (p_0)_{r_1,s_1}'(\cdot)}\\ 
  &= 
  \|\mu\chi_Q\|_{e_0(\cdot)}\|\mu^{-1}\chi_Q\|_{ e_0'(\cdot)}.
\end{align*}
Thus we obtain that 
\begin{align*}
  [w_0]_{_{(p_0(\cdot),q_0(\cdot)),(\vec{r},\vec{s})}}= [\mu]_{\mathcal{A}_{e_0(\cdot)}}^{1/r_1-1/s_1} < \infty,
\end{align*}
which completes the proof.
\end{proof}


\section{Proofs of Theorems \ref{thm:main1} and \ref{thm:main2}}\label{sect:last}

We begin by recalling some interpolation results for compact operators and variable exponent Lebesgue spaces. 
The first is the Cwikel--Kalton \cite{CwKa} complex interpolation for compact linear operators. Here, we recall the version as stated in \cite[Theorem 3.1]{HL}.

\begin{theorem}\label{thm:CwKa}
Let $(X_0,X_1)$ and $(Y_0,Y_1)$ be Banach couples and $T$ be a linear operator such that
$T:X_0+X_1\to Y_0+Y_1$ and $T:X_i\to Y_i$ boundedly for $i=0,1$.
Suppose moreover that $T:X_1\to Y_1$ is compact.
Let $[\ ,\ ]_\theta$ be the complex interpolation functor of Calder\'on.
Then also $T:[X_0,X_1]_\theta\to[Y_0,Y_1]_\theta$ is compact for $\theta\in(0,1)$ under any of the following four side conditions:

\begin{enumerate}
  \item\label{it:UMD} $X_1$ has the UMD (unconditional martingale differences) property,
  \item\label{it:Xinterm} $X_1$ is reflexive, and $X_1=[X_0,E]_\alpha$ for some Banach space $E$ and $\alpha\in(0,1)$,
  \item\label{it:Yinterm} $Y_1=[Y_0,F]_\beta$ for some Banach space $F$ and $\beta\in(0,1)$,
  \item\label{it:lattice} $X_0$ and $X_1$ are both complexified Banach lattices of measurable functions on a common measure space.
\end{enumerate}
\end{theorem}

\begin{lemma}\label{lem:CK condition}
If $p_i(\cdot)\in\mathcal{P}$ with $1<(p_{i})_{-}\leq(p_{i})_{+}<\infty$ and $w_i$ are weights satisfying $w_i\in L_{\text{loc}}^{p_i(\cdot)}$ and $w_i^{-1}\in L_{\text{loc}}^{p_{i}'(\cdot)}$, then the spaces $X_i=L^{p_i(\cdot)}_{w_i}$ satisfy the condition \eqref{it:lattice} of Theorem \ref{thm:CwKa}.
\end{lemma}

\begin{proof}
It is known, e.g. from \cite[page 565]{INS}, 
that both $X_i=L^{p_i(\cdot)}_{w_i}$, for $i=0,1$, are Banach function spaces. Hence, they are also complexified Banach lattices of measurable functions on the common measure space $\R^d$.
\end{proof}

Another basic ingredient is interpolation with the change of weights. In the case of weighted variable Lebesgue spaces the following interpolation result is proved in \cite{Medalha} (see also \cite{FKM2022}).

\begin{theorem}[\cite{Medalha}, Theorems 3.4.1 and 3.4.2]\label{thm:SW}
Let $0<\theta<1$. For $i=0,1$, let $q_i(\cdot)\in\mathcal{P}$ be variable exponents satisfying $1<(q_i)_{-}\leq(q_i)_{+}<\infty$ and $w_i$ be weights satisfying $w_i\in L_{\text{loc}}^{q_i(\cdot)}$ and $w_i^{-1}\in L_{\text{loc}}^{q_{i}'(\cdot)}$.
Then 
\begin{equation*}
  [L^{q_0(\cdot)}_{w_0},L^{q_1(\cdot)}_{w_1}]_{\theta}=L^{q_\theta(\cdot)}_{w_{\theta}},
\end{equation*}
where 
\begin{equation*}
  \frac{1}{q_{\theta}(\cdot)}=\frac{1-\theta}{q_0(\cdot)}+\frac{\theta}{q_1(\cdot)}\quad\text{and}\quad w_{\theta}=w_0^{1-\theta}w_1^{\theta}.
\end{equation*}
Moreover, the norms of the spaces $[L^{q_0(\cdot)}_{w_0},L^{q_1(\cdot)}_{w_1}]_{\theta}$ and $L^{q_\theta(\cdot)}_{w_{\theta}}$ are equivalent. 
\end{theorem}

\begin{remark}\label{rmk:CK condition}
Notice that the weights belonging to the off-diagonal full/limited range variable classes as in Definitions \ref{def:variableApq}/\ref{variableAp,r,s} satisfy the assumptions and conclusions of Lemma \ref{lem:CK condition} and Theorem \ref{thm:SW}.
\end{remark}

We now turn to proving Theorems \ref{thm:main1} and \ref{thm:main2} (the latter automatically yields the first as a special case, in fact) by combining Theorem \ref{thm:fact.main2} with Theorems \ref{thm:CwKa} and \ref{thm:SW}.

\begin{proof}[Proof of Theorem \ref{thm:main2}]
Fix some $\big(p(\cdot),q(\cdot)\big) \in \vec{\mathcal{E}}_{(\vec{r},\vec{s})}^{\gamma}$ and $w\in\mathcal{A}_{(p(\cdot),q(\cdot)), (\vec{r},\vec{s})}$. We will show that $T$ is $L^{p(\cdot)}_w\to L^{q(\cdot)}_w$ compact. 
By Theorem \ref{thm:fact.main2}, there exists some $\eta\in(0,1)$ so that for all $\theta\in(0,\eta)$ there holds that 
$\big(p_0(\cdot),q_0(\cdot)\big)\in\vec{\mathcal{E}}_{(\vec{r},\vec{s})}^{\gamma}$ and $w_0\in\mathcal{A}_{(p_0(\cdot),q_0(\cdot)), (\vec{r},\vec{s})},$ where these objects are defined through the relations
\begin{equation*}
  \frac{1}{p(\cdot)}=\frac{1-\theta}{p_0(\cdot)}+\frac{\theta}{p_1(\cdot)},\qquad\frac{1}{q(\cdot)}=\frac{1-\theta}{q_0(\cdot)}+\frac{\theta}{q_1(\cdot)},\qquad
  w=w_0^{1-\theta}w_1^{\theta}.
\end{equation*}
Now we fix any such $\theta \in (0,\eta)$ (e.g. $\theta = \eta/2$) and then by Theorem \ref{thm:SW} (applicable due to Remark \ref{rmk:CK condition})
there holds that
\begin{equation*}
    \begin{split}
        [L^{p_0(\cdot)}_{w_0},L^{p_1(\cdot)}_{w_1}]_\theta=L^{p(\cdot)}_w,\qquad 
        [L^{q_0(\cdot)}_{w_0},L^{q_1(\cdot)}_{w_1}]_\theta=L^{q(\cdot)}_w.
    \end{split}
\end{equation*}

It remains to interpret everything into the context of Theorem \ref{thm:CwKa}.
We let $X_i:=L^{p_i(\cdot)}_{w_i}$ and $Y_i:=L^{q_i(\cdot)}_{w_i}.$ We know that $T$ maps $X_0+X_1\to Y_0+Y_1$ (structural assumption on $T$), $T:X_0\to Y_0$ is bounded and $T: X_1\to Y_1$ is compact (both are assumed). By Lemma \ref{lem:CK condition} and Remark \ref{rmk:CK condition}, the condition \eqref{it:lattice} of Theorem \ref{thm:CwKa} is satisfied by $X_i=L^{p_i(\cdot)}_{w_i},$ for $i=0,1.$ Thus, by Theorem \ref{thm:CwKa}, it follows that $T$ is $L^{p(\cdot)}_w = [X_0,X_1]_\theta \to[Y_0,Y_1]_\theta=L^{q(\cdot)}_w$ compact, just as we wanted to prove.
\end{proof}

\subsection*{Funding} S. L. was supported by the Primus research programme PRIMUS/21/SCI/002 of Charles University and the Foundation for Education and European Culture, founded by Nicos and Lydia Tricha. 
T.O. was supported by the Research Council of Finland through Project 358180.

\section*{Declarations}

\subsection*{Conflict of interest} The authors have no relevant financial or non-financial interests to disclose.


\begin{thebibliography}{99}





\bibitem{CMM}
M.~Cao, J.J.~Mar\'in and J. M.~Martell.
\newblock Extrapolation on function and modular spaces, and applications.
\newblock {\em Adv. Math.}, {\bf 406} (2022), Paper No.~108520, 87 pp.

\bibitem{COY}
M.~Cao, A.~Olivo and K.~Yabuta.
\newblock Extrapolation for multilinear compact operators and applications.
\newblock {\em Trans. Amer. Math. Soc.}, {\bf 375} (2022), no.~7, 5011--5070.





\bibitem{CF}
D.~Cruz-Uribe and A.~Fiorenza.
\newblock Variable Lebesgue spaces, Applied and Numerical Harmonic Analysis.
\newblock Birkh\"auser/Springer, Heidelberg, 2013. 
\newblock Foundations and harmonic analysis.




\bibitem{CP1}
D.~Cruz-Uribe and M.~Penrod.
\newblock The reverse H{\"o}lder inequality for $\mathcal{A}_p(\cdot)$ weights with applications to matrix weights. 
\newblock {\em Publ. Mat.}, {\bf 70} (2026), no.~2, 493--524.



\bibitem{CW2017}
D.~Cruz-Uribe and L. A. D.~Wang.
\newblock Extrapolation and weighted norm inequalities in the variable Lebesgue spaces.
\newblock {\em Trans. Amer. Math. Soc.}, {\bf 369} (2017), no.~2, 1205--1235.


\bibitem{CwKa}
M.~Cwikel and N.~J. Kalton.
\newblock Interpolation of compact operators by the methods of {C}alder\'on and {G}ustavsson-{P}eetre.
\newblock {\em Proc. Edinburgh Math. Soc.}, {\bf 38} (1995), no.~2, 261--276.

\bibitem{DHHR}
L.~Diening, P.~Harjulehto, P.~H{\"a}st{\"o} and M.~Ru\v{z}i\v{c}ka.
\newblock Lebesgue and Sobolev spaces with variable exponents. 
\newblock Lecture Notes in Mathematics, vol. 2017, Springer, Heidelberg, 2011.


\bibitem{FKM2022}
C.~Fernandes, O.~Karlovych and S.~Medalha. 
\newblock Invertibility of Fourier convolution operators with $PC$ symbols on variable Lebesgue spaces with Khvedelidze weights.
\newblock {\em J. Math. Sci. (N.Y.)}, {\bf 266} (2022), no.~3, 419--434.

\bibitem{HL2022}
T.~Hyt{\"o}nen and S.~Lappas.
\newblock Extrapolation of compactness on weighted spaces: Bilinear operators. 
\newblock {\em Indag. Math. (N.S.)}, {\bf 33} (2022), no.~2, 397--420.

\bibitem{HL}
T.~Hyt{\"o}nen and S.~Lappas.
\newblock Extrapolation of compactness on weighted spaces.
\newblock {\em Rev. Mat. Iberoam.}, {\bf 39} (2023), no.~1, 91--122.


\bibitem{INS}
M.~Izuki, E.~Nakai and Y.~Sawano.
\newblock Wavelet characterization and modular inequalities for weighted Lebesgue spaces with variable exponent.
\newblock {\em Ann. Acad. Sci. Fenn. Math.}, {\bf 40} (2015), no.~2, 551--571.



\bibitem{LWY2023}
S.~Liu, H.~Wu and D.~Yang. 
\newblock A note on extrapolation of compactness.
\newblock {\em Collect. Math.}, {\bf 74} (2023), no.~2, 375--390.

\bibitem{LN}
E.~Lorist and Z.~Nieraeth.
\newblock Extrapolation of compactness on Banach function spaces.
\newblock {\em J. Fourier Anal. Appl.}, {\bf 30} (2024), no.~3, Paper No.~30, 25 pp.


\bibitem{RR2025}
F. J.~Mart\'{i}n-Reyes and I. P.~Rivera-R\'{i}os.
\newblock A note on one-sided extrapolation of compactness and applications.
\newblock {\em Studia Math.}, {\bf 284} (2025), no.~1, 39--68.

\bibitem{Medalha}
S. J. B.~Medalha.
\newblock Algebras of Convolution Type Operators on Weighted Variable Lebesgue Spaces.
\newblock Master Thesis, NOVA University Lisbon,
\newblock \href{https://run.unl.pt/bitstream/10362/135865/1/Medalha_2021.pdf}{Medalha\_2020.pdf}, 2020.


\bibitem{N2023}
Z.~Nieraeth.
\newblock Extrapolation in general quasi-Banach function spaces.
\newblock {\em J. Funct. Anal.}, {\bf 285} (2023), no.~10, Paper No.~110130, 109 pp.

\bibitem{RdF}
J.~L. Rubio de Francia.
\newblock Factorization theory and  $A_p$  weights.
\newblock {\em Amer. J. Math.}, {\bf 106} (1984), no.~3, 533--547.


\bibitem{TWX}
J.~Tan, J.~Wang and Q.~Xue.
\newblock Boundedness of a class of multilinear operators and their iterated commutators on Morrey-Banach function spaces.
\newblock {\em J. Geom. Anal.}, {\bf 35} (2025), no.~7, Paper No.~208, 50 pp.



\end{thebibliography}
\end{document}